\documentclass[11pt]{amsart}

\usepackage{amscd,amssymb,bm,enumitem,mleftright,pgf,tikz,xcolor,hyperref}
\usepackage[margin = 1in]{geometry}

\setlist[itemize]{leftmargin = *}
\setlist[enumerate]{leftmargin = *}

\allowdisplaybreaks

\newtheorem{Cor}{Corollary}[section]

\newtheorem{Thm}[Cor]{Theorem}
\newtheorem{Prop}[Cor]{Proposition}
\newtheorem{rmk}[Cor]{Remark}
\theoremstyle{definition}
\newtheorem{Def}[Cor]{Definition}

\theoremstyle{remark}
\newtheorem{Eg}[Cor]{Example}

\newcommand{\C}{\mathbb{C}}

\newcommand{\N}{\mathbb{N}}
\newcommand{\R}{\mathbb{R}}
\newcommand{\T}{\mathbb{T}}
\newcommand{\X}{\mathsf{X}}
\newcommand{\Y}{\mathsf{Y}}
\newcommand{\Z}{\mathbb{Z}}
\newcommand{\Br}[1]{\mleft( #1 \mright)}
\newcommand{\Cc}[1]{\func{C_{c}}{#1}}

\newcommand{\df}{\stackrel{\textnormal{df}}{=}}
\newcommand{\ds}{\displaystyle}
\newcommand{\Id}{\operatorname{Id}}
\newcommand{\Lp}[2]{\func{L^{#1}}{#2}}

\newcommand{\Act}[3]{\func{#1_{#2}}{#3}}
\newcommand{\Adj}[1]{\func{\mathcal{L}}{#1}}
\newcommand{\Aut}[1]{\func{\operatorname{Aut}}{#1}}

\newcommand{\Int}[3]{\int_{#1} #2 ~ \mathrm{d}{#3}}
\newcommand{\Map}[4]{\mleft\{ \begin{matrix} #1 & \to & #2 \\ \ds #3 & \mapsto & \ds #4 \end{matrix} \mright\}}
\newcommand{\red}{\operatorname{red}}
\newcommand{\Seq}[2]{\Br{#1}_{#2}}

\newcommand{\Comp}[1]{\func{\mathcal{K}}{#1}}

\newcommand{\func}[2]{#1 \Br{#2}}
\newcommand{\Func}[2]{\func{\Br{#1}}{#2}}
\newcommand{\FUNC}[2]{\func{\SqBr{#1}}{#2}}
\newcommand{\Mult}[2]{\mathsf{M}_{#1,#2}}
\newcommand{\MULT}[2]{\overline{\mathsf{M}}_{#1,#2}}
\newcommand{\Norm}[2]{\mleft\| #1 \mright\|_{#2}}
\newcommand{\Pair}[2]{\Br{#1,#2}}
\newcommand{\Span}[1]{\func{\operatorname{Span}}{#1}}
\newcommand{\SqBr}[1]{\mleft[ #1 \mright]}
\newcommand{\SSet}[1]{\mleft\{ #1 \mright\}}
\newcommand{\Supp}[1]{\func{\operatorname{Supp}}{#1}}
\newcommand{\Thet}[3]{\Theta^{#1}_{#2,#3}}
\newcommand{\Toep}[1]{\func{\mathcal{T}}{#1}}
\newcommand{\Cstar}{C^{\ast}}
\newcommand{\Inner}[3]{\mleft\langle #1 \middle| #2 \mright\rangle_{#3}}
\newcommand{\InvIm}[2]{#1^{- 1} \SqBr{#2}}
\newcommand{\Range}[1]{\func{\operatorname{Range}}{#1}}

\renewcommand{\O}[1]{\func{\mathcal{O}}{#1}}
\renewcommand{\Im}[2]{#1 \SqBr{#2}}

\title{Group Actions on Product Systems}

\author{Valentin Deaconu}
\address{Valentin Deaconu \\ Department of Mathematics \& Statistics\\ University of Nevada \\ Reno, NV 89557-0084 \\ USA}
\email{vdeaconu@unr.edu}

\author{Leonard Huang}
\address{Leonard Huang \\ Department of Mathematics \& Statistics\\ University of Nevada \\ Reno, NV 89557-0084 \\ USA}
\email{LeonardHuang@unr.edu}



\date{December 2022}

\begin{document}



\begin{abstract}
We introduce the concept of  crossed product of a product system by a locally compact group. We prove that the crossed product of a row-finite and faithful product system by an amenable group is also a row-finite and faithful product system. We illustrate with examples related to group actions on $k$-graphs and to higher rank Doplicher-Roberts algebras.

\end{abstract}


\maketitle



\section{Introduction}


Product systems over various discrete semigroups were introduced by N. Fowler in \cite{F}, inspired by work of W. Arveson and studied by several authors (see \cite{AM,CLSV,SY}, for example). Several interesting examples of product systems already occur over the semigroup $ \Pair{\N^{k}}{+} $, where $ k \geq 2 $.

We first recall the Toeplitz algebra and the Cuntz-Pimsner algebra of a product system. We use the covariance condition in Fowler's sense. Next, we introduce the concept of an action of a (locally compact and Hausdorff) group on a product system and then define the associated crossed product product system. We prove that the crossed product of a row-finite and faithful product system by an amenable group is also row-finite and faithful, and, furthermore, we establish a version of the Hao-Ng Theorem for product systems in our setting.

Motivations come from two sources: (i) group actions on higher-rank graphs; (ii) the higher rank Doplicher-Roberts algebra defined from $ k $ representations of a compact group.
\section{$ \Cstar $-Algebras of Product Systems}


Let us first recall the definition of a product system. Let $ \Pair{P}{\cdot} $ be a discrete monoid with identity $ e $, and let $ A $ be a $ \Cstar $-algebra. A $ P $-indexed \emph{product system} of $ \Cstar $-correspondences over $ A $ is a semigroup $ \ds \Y = \bigsqcup_{p \in P} \Y_{p} $ with the following properties:
\begin{itemize}
\item
For each $ p \in P $, the object $ \Y_{p} $ is a $ C^*$-correspondence over $ A$, which we call the \emph{fiber} of $ \Y $ at $ p $. Its inner product is denoted by $\Inner{\cdot}{\cdot}{\Y_p}$.

\item
The fiber $ \Y_{e} $ of $ \Y $ at $ e $ is $ _{A} A_{A} $, which is $ A $ viewed as an $ A $-correspondence over itself.

\item
For each $ p,q \in P $, the semigroup multiplication on $ \Y $ maps  $ \Y_{p} \times \Y_{q} $ to $ \Y_{p q} $, so we have an $ A $-balanced $ \C $-bilinear map
$$
\Mult{p}{q} \df \Map{\Y_{p} \times \Y_{q}}{\Y_{p q}}{\Pair{x}{y}}{x \cdot y}.
$$

\item
For each $ p,q \in P \setminus \SSet{e} $, the map $ \Mult{p}{q}: \Y_{p} \times \Y_{q} \to \Y_{p q} $ induces an isomorphism $ \MULT{p}{q}: \Y_{p} \otimes_{A} \Y_{q} \to \Y_{p q} $.

\item
For each $ p \in P $, the maps $ \Mult{e}{p} $ and $ \Mult{p}{e} $ implement, respectively, the left and right actions of $ A $ on $ \Y_{p} $. Consequently, $ \MULT{p}{e}: \Y_{p} \otimes_{A} \Br{_{A} A_{A}} \to \Y_{p} $ is an isomorphism for all $ p \in P $.
\end{itemize}

For each $ p \in P $, let $ \phi_{p}: A \to \Adj{\Y_{p}} $ denote the left action of $ A $ on $ \Y_{p} $ by adjointable operators. We say that $ \Y $ is \emph{essential} if and only if $ \Y_{p} $ is an essential $ A $-correspondence, i.e., $ \Span{\Im{\Im{\phi_{p}}{A}}{\Y_{p}}} $ is dense in $ \Y_{p} $, for each $ p \in P $, in which case $ \MULT{e}{p}: \Br{_{A} A_{A}} \otimes_{A} \Y_{p} \to \Y_{p} $ is also an isomorphism.

If $ \phi_{p} $ takes values in the $ C^{\ast} $-algebra $ \Comp{\Y_{p}} $ of compact operators on $ \Y_{p} $ for each $ p \in P $, then $ \Y $ is said to be \emph{row-finite} or \emph{proper}, and if $ \phi_{p} $ is furthermore injective for each $ p \in P $, then $ \Y $ is said to be \emph{faithful}.

There are various $ \Cstar $-algebras associated to a product system under certain assumptions. For our future reference, let us recall some standard facts. 

Let $ \Y $ be a $ P $-indexed product system over  $ A $, and let $ B $ be a $ \Cstar $-algebra. A map $ \psi: \Y \to B $ is then called a \emph{Toeplitz representation} of $ \Y $ if and only if, writing $ \psi_{p} $ for $ \psi|_{\Y_{p}} $, the following properties hold:
\begin{itemize}
\item
$ \psi_{p}: \Y_{p} \to B $ is $ \C $-linear for all $ p \in P $.

\item
$ \psi_{e}: A \to B $ is a $ \Cstar $-homomorphism, and $ \func{\psi_{e}}{\Inner{\zeta}{\eta}{\Y_{p}}} = \func{\psi_{p}}{\zeta}^{\ast} \func{\psi_{p}}{\eta} $ for all $ p \in P $ and $ \zeta,\eta \in \Y_{p} $.

\item
$ \func{\psi_{p}}{\zeta} \func{\psi_{q}}{\eta} = \func{\psi_{p q}}{\zeta \eta} $ for all $ p,q \in P $, $ \zeta \in \Y_{p} $, and $ \eta \in \Y_{q} $.
\end{itemize}
One can construct a $ C^{\ast} $-algebra $ \Toep{\Y} $ --- known as the \emph{Toeplitz algebra} of $ \Y $ --- and a Toeplitz representation $ i_{\Y}: \Y \to \Toep{\Y} $ of $ \Y $ such that the pair $ \Pair{\Toep{\Y}}{i_{\Y}} $ is universal in the following sense: $ \Toep{\Y} $ is generated by $ \Im{i_{\Y}}{\Y} $, and for any Toeplitz representation $ \psi: \Y \to B $, there is a $ \Cstar $-homomorphism $ \psi_{\ast}: \Toep{\Y} \to B $ such that $ \psi_{\ast} \circ i_{\Y} = \psi $.

For each $ p \in P $, there exists a $ \Cstar $-homomorphism $ \psi^{\Br{p}}: \Comp{\Y_{p}} \to B $ obtained as the continuous extension of the map
$$
\forall \zeta_{1},\ldots,\zeta_{n},\eta_{1},\ldots,\eta_{n} \in \Y_{p}: \qquad
\sum_{i = 1}^{n} \Thet{}{\zeta_{i}}{\eta_{i}} \mapsto \sum_{i = 1}^{n} \func{\psi_{p}}{\zeta_{i}} \func{\psi_{p}}{\eta_{i}}^{\ast}.
$$
Here $ \Thet{}{\zeta}{\eta}(\xi)=\zeta\Inner{\eta}{\xi}{\Y_p}$. Note that $ \psi^{\Br{e}} = \psi_{e} $. 

A Toeplitz representation $ \psi: \Y \to B $ is then called \emph{Cuntz-Pimsner covariant} (in Fowler's sense) if and only if
$$
\forall p \in P, ~ \forall a \in \InvIm{\phi_{p}}{\Comp{\Y_{p}}}: \qquad
\func{\psi^{\Br{p}}}{\func{\phi_{p}}{a}} = \func{\psi_{e}}{a}.
$$
One can construct a $ C^{\ast} $-algebra $ \O{\Y} $ --- known as the \emph{Cuntz-Pimsner algebra} of $ \Y $ --- and a Cuntz-Pimsner covariant Toeplitz representation $ j_{\Y}: \Y \to \O{\Y} $ of $ \Y $ such that the pair $ \Pair{\O{\Y}}{j_{\Y}} $ is universal in the following sense: $ \O{\Y} $ is generated by $ \Im{j_{\Y}}{\Y} $, and for any Cuntz-Pimsner covariant Toeplitz representation $ \psi: \Y \to B $, there is a $ C^{\ast} $-homomorphism $ \psi_{\ast}: \O{\Y} \to B $ such that $ \psi_{\ast} \circ j_{\Y} = \psi $.

\begin{Eg}\label{Fow}
For  a   product system $\Y\to P$ with fibers $\Y_p$  nonzero finitely dimensional Hilbert spaces, in particular $A=\Y_e=\C$, let us fix an orthonormal basis $\mathcal B_p$ in $\Y_p$. Then a Toeplitz representation $\psi:\Y\to B$ gives rise to a family of isometries $\{\psi(\xi): \xi\in \mathcal B_p\}_{p\in P}$ with mutually orthogonal range projections. In this case $\Toep{\Y}$ is generated by a colection of Cuntz-Toeplitz algebras which interact according to the multiplication maps $\MULT{p}{q}$  in $\Y$.

A representation $\psi:\Y\to B$ is  Cuntz-Pimsner covariant  if \[ \sum_{\xi\in \mathcal B_p}\psi(\xi)\psi(\xi)^*=\psi(1)\] for all $p\in P$. 
The Cuntz-Pimsner algebra $\O{\Y}$ is generated by a collection of Cuntz algebras, so it could be thought as a multidimensional Cuntz algebra.  N. Fowler proved in \cite{F1} that if the  function $p\mapsto \dim \Y_p$ is injective, then the algebra $\O{\Y}$ is simple and purely infinite. For other examples of multidimensional Cuntz algebras, see  \cite{ B}.   
\end{Eg}
\begin{Eg}\label{k-graph}
A row-finite  $k$-graph with no sources $\Lambda$ (see \cite{KP}) determines a product system $\Y\to \N^k$ with $\Y_{0}=A=C_0(\Lambda^0)$ and $\Y_n=\overline{C_c(\Lambda^n)}$ for $n\neq 0$ such that we have an isomorphism $\O{\Y}\cong C^*(\Lambda)$. 

\end{Eg}



\section{Group Actions on Product Systems and Crossed Products}


Given a locally compact group $ G $ and an $C^*$-correspondence $ \X $ over $A$, recall that an action of $ G $ on $ \X $ is a pair $ \Pair{\alpha}{\beta} $ with the following properties:
\begin{itemize}
\item
$ \alpha $ is a strongly continuous action of $ G $ on $ A $ by $ \Cstar $-automorphisms.

\item
$ \beta $ is a strongly continuous action of $ G $ on $ \X $ by surjective $ \C $-linear isometries.

\item
For all $ s \in G $, $ a \in A $, and $ x,y \in \X $,
$$
\Inner{\Act{\beta}{s}{x}}{\Act{\beta}{s}{y}}{\X} = \Act{\alpha}{s}{\Inner{x}{y}{\X}},    \qquad
\Act{\beta}{s}{x a}                              = \Act{\beta}{s}{x} \Act{\alpha}{s}{a}, \qquad
\Act{\beta}{s}{a x}                              = \Act{\alpha}{s}{a} \Act{\beta}{s}{x}.
$$
\end{itemize}

By the universal property of Cuntz-Pimsner algebras defined using the Katsura ideal, see \cite{K}, there is an action $ \gamma $ of $ G $ on $ \O{\X} $.

The crossed product $ \X \rtimes_{\beta} G $ of $ \X $ by $ G $ is defined as the completion of the $ \Cc{G,A} $-bimodule $ \Cc{G,\X} $, and its $ \Br{A \rtimes_{\alpha} G} $-correspondence structure is uniquely determined by the following operations:
\begin{gather*}
\forall f \in \Cc{G,A}, ~ \forall \zeta,\eta \in \Cc{G,\X}, ~ \forall s \in G: \\
\Func{f \zeta}{s} = \Int{G}{\func{f}{t} \Act{\beta}{t}{\func{\zeta}{t^{- 1} s}}}{t}, \quad
\Func{\zeta f}{s} = \Int{G}{\func{\zeta}{t} \Act{\alpha}{t}{\func{f}{t^{- 1} s}}}{t}, \\
\func{\Inner{\zeta}{\eta}{\X \rtimes_{\beta} G}}{s} = \Int{G}{\Act{\alpha}{t^{- 1}}{\Inner{\func{\zeta}{t}}{\func{\eta}{t s}}{\X}}}{t}.
\end{gather*}
For $ G $ amenable, it is proven in \cite{HN} that
$$
\O{\X} \rtimes_{\gamma} G \cong \O{\X \rtimes_{\beta} G}.
$$



\begin{Def}
An action $ \beta $ of a locally compact group $ G $ on a product system $ \Y \to P $ over $A$ is a $ P $-indexed family $ \Seq{\beta^{p}}{p \in P} $ such that $ \Pair{\beta^{e}}{\beta^{p}} $ is an action of $ G $ on $ \Y_{p} $ for each $ p \in P $, and furthermore,
$$
\forall s \in G, ~ \forall \zeta \in \Y_{p}, ~ \forall \eta \in \Y_{q}: \qquad
\Act{\beta^{p q}}{s}{\zeta \eta} = \Act{\beta^{p}}{s}{\zeta} \Act{\beta^{q}}{s}{\eta}.
$$
We will usually denote $ \beta^{e} $ by $ \alpha $.
\end{Def}



\begin{Eg}
For an essential product system $ \Y $ indexed by  $ P = \Pair{\N^{k}}{+} $ such that $ \phi_{p} $ is an injection into $ \Comp{\Y_{p}} $ for all $ p \in \N^{k} $, universality allows us to define a strongly continuous gauge action $ \gamma: \T^{k} \to \Aut{\O{\Y}} $ such that
$$
\forall z \in \T^{k}, ~ \forall p \in \N^{k}, ~ \forall a \in A, ~ \forall \zeta \in \Y_{p}: \qquad
\Act{\gamma}{z}{a} = a
\qquad \text{and} \qquad
\Act{\gamma}{z}{\func{j_{\Y}}{\zeta}} = z^{p} \func{j_{\Y}}{\zeta}.
$$
Here, $ \ds z^{p} \df \prod_{i = 1}^{k} z_{i}^{p_{i}} $. Then the fixed-point algebra $ \O{\Y}^{\gamma} $ is $ \Cstar $-isomorphic to the inductive limit
$$
\varinjlim_{p \in \N^{k}} \Comp{\Y_{p}},
$$
where the order relation on $ \N^{k} $ is the coordinatewise order.
\end{Eg}

\begin{Eg}
For a compact group $G$ and $k$ finite dimensional unitary representations $\rho_i$ of $G$ on Hilbert spaces $\mathcal H_i$ for $i=1,...,k$, we can construct  a product system $\Y$ with fibers 
\[\Y_n=\mathcal H_1^{\otimes n_1}\otimes\cdots\otimes  \mathcal H_k^{\otimes n_k}\] 
for $n=(n_1,...,n_k)\in \N^k$, see \cite{D}. Then the group $G$ acts on each fiber $\Y_n$  via the representation $\rho^n=\rho_1^{\otimes n_1}\otimes\cdots\otimes \rho_k^{\otimes n_k}$. This action is compatible with the multiplication maps and commutes with the gauge action of $\T^k$.
\end{Eg}



\begin{Def} \label{cp}
If $ \beta $ is an action of $ G $ on a $ P $-indexed product system $ \Y $, then the \emph{crossed product} $ \Y \rtimes_{\beta} G $ is defined as the $ P $-indexed product system with fibers $ \Y_{p} \rtimes_{\beta^{p}} G $, which are $C^*$-correspondences over $ A \rtimes_{\alpha} G $. For $ \zeta \in \Cc{G,\Y_{p}} $ and $ \eta \in \Cc{G,\Y_{q}} $, the product $ \zeta \eta \in \Cc{G,\Y_{p q}} $ is given by
$$
\forall s \in G: \qquad
\Func{\zeta \eta}{s} = \Int{G}{\func{\zeta}{t} \Act{\beta^{q}}{t}{\func{\eta}{t^{- 1} s}}}{t}.
$$
\end{Def}



\begin{Prop}
The semigroup $ \ds \Y \rtimes_{\beta} G = \bigsqcup_{p \in P} \Br{\Y_{p} \rtimes_{\beta^{p}} G} $ with the above multiplication law  satisfies all the properties of a product system over $ A \rtimes_{\alpha} G $.
\end{Prop}

\begin{proof}
Let us first prove that the multiplication law for $ \Y \rtimes_{\beta} G $ is associative on the function-algebra level. Let $ p,q,r \in P $, and let $ \zeta \in \Cc{G,\Y_{p}} $, $ \eta \in \Cc{G,\Y_{q}} $, and $ \xi \in \Cc{G,\Y_{r}} $. Then for all $ s \in G $,
\begin{align*}
    \FUNC{\Br{\zeta \eta} \xi}{s}
& = \Int{G}{\Func{\zeta \eta}{t} \Act{\beta^{r}}{t}{\func{\xi}{t^{- 1} s}}}{t} \\
& = \Int{G}{\SqBr{\Int{G}{\func{\zeta}{u} \Act{\beta^{q}}{u}{\func{\eta}{u^{- 1} t}}}{u}} \Act{\beta^{r}}{t}{\func{\xi}{t^{- 1}} s}}{t} \\
& = \Int{G \times G}{\SqBr{\func{\zeta}{u} \Act{\beta^{q}}{u}{\func{\eta}{u^{- 1} t}}} \Act{\beta^{r}}{t}{\func{\xi}{t^{- 1}} s}}
        {\Br{u \times t}} \\
& = \Int{G \times G}{\func{\zeta}{u} \SqBr{\Act{\beta^{q}}{u}{\func{\eta}{u^{- 1} t}} \Act{\beta^{r}}{t}{\func{\xi}{t^{- 1}} s}}}
        {\Br{u \times t}}
\end{align*}
and
\begin{align*}
    \FUNC{\zeta \Br{\eta \xi}}{r}
& = \Int{G}{\func{\zeta}{t} \Act{\beta^{q r}}{t}{\Func{\eta \xi}{t^{- 1} r}}}{t} \\
& = \Int{G}{\func{\zeta}{t} \Act{\beta^{q r}}{t}{\Int{G}{\func{\eta}{s} \Act{\beta^{r}}{s}{\func{\xi}{s^{- 1} t^{- 1} r}}}{s}}}{t} \\
& = \Int{G \times G}{\func{\zeta}{t} \Act{\beta^{q r}}{t}{\func{\eta}{s} \Act{\beta^{r}}{s}{\func{\xi}{s^{- 1} t^{- 1} r}}}}
        {\Br{s \times t}} \\
& = \Int{G \times G}{\func{\zeta}{s} \Act{\beta^{q r}}{s}{\func{\eta}{t} \Act{\beta^{r}}{t}{\func{\xi}{t^{- 1} s^{- 1} r}}}}
        {\Br{s \times t}} \\
& = \Int{G \times G}{\func{\zeta}{s} \Act{\beta^{q}}{s}{\func{\eta}{t}} \Act{\beta^{r}}{s t}{\func{\xi}{t^{- 1} s^{- 1} r}}}
        {\Br{s \times t}} \qquad \Br{\text{By the axioms for a group action.}} \\
& = \Int{G \times G}{\func{\zeta}{s} \Act{\beta^{q}}{s}{\func{\eta}{s^{- 1} t}} \Act{\beta^{r}}{t}{\func{\xi}{t^{- 1} r}}}{\Br{s \times t}}.
    \qquad \Br{\text{By the change of variables $ t \mapsto s^{- 1} t $.}}
\end{align*}
It follows that for all $ p,q \in P $ 
$$
\Map{\Cc{G,\Y_{p}} \times \Cc{G,\Y_{q}}}{\Cc{G,\Y_{p q}}}{\Pair{\zeta}{\eta}}{\zeta \eta}
$$
is a $ \Cc{G,A} $-balanced $ \C $-bilinear map, which then induces a $ \C $-linear map
$$
  \Omega_{p,q}
= \Map{\Cc{G,\Y_{p}} \otimes_{\Cc{G,A}} \Cc{G,\Y_{q}}}{\Cc{G,\Y_{p q}}}{\sum_{i = 1}^{n} \zeta_{i} \odot \eta_{i}}
      {\sum_{i = 1}^{n} \zeta_{i} \eta_{i}}.
$$
Let us show that $ \Omega_{p,q} $ extends uniquely to a $ \C $-linear isometry
$$
\overline{\Omega}_{p,q}:
\Br{\Y_{p} \rtimes_{\beta^{p}} G} \otimes_{A \rtimes_{\alpha} G} \Br{\Y_{q} \rtimes_{\beta^{q}} G} \to
\Y_{p q} \rtimes_{\beta^{p q}} G.
$$
Observe that for all $ \zeta_{1},\ldots,\zeta_{n} \in \Cc{G,\Y_{p}} $ and $ \eta_{1},\ldots,\eta_{n} \in \Cc{G,\Y_{q}} $ we have
\begin{align*}
  & ~ \Norm{\sum_{i = 1}^{n} \zeta_{i} \otimes \eta_{i}}
           {\Br{\Y_{p} \rtimes_{\beta^{p}} G} \otimes_{A \rtimes_{\alpha} G} \Br{\Y_{q} \rtimes_{\beta^{q}} G}} \\
= & ~ \Norm{
           \Inner{\sum_{i = 1}^{n} \zeta_{i} \otimes \eta_{i}}
                 {\sum_{j = 1}^{n} \zeta_{j} \otimes \eta_{j}}
                 {\Br{\Y_{p} \rtimes_{\beta^{p}} G} \otimes_{A \rtimes_{\alpha} G} \Br{\Y_{q} \rtimes_{\beta^{q}} G}}
           }
           {A \rtimes_{\alpha} G}^{\frac{1}{2}} \\
= & ~ \Norm{
           \sum_{i,j = 1}^{n}
           \Inner{\zeta_{i} \otimes \eta_{i}}{\zeta_{j} \otimes \eta_{j}}
                 {\Br{\Y_{p} \rtimes_{\beta^{p}} G} \otimes_{A \rtimes_{\alpha} G} \Br{\Y_{q} \rtimes_{\beta^{q}} G}}
           }
           {A \rtimes_{\alpha} G}^{\frac{1}{2}} \\
= & ~ \Norm{
           \sum_{i,j = 1}^{n}
           \Inner{\eta_{i}}{\Inner{\zeta_{i}}{\zeta_{j}}{\Y_{p} \rtimes_{\beta^{p}} G} \eta_{j}}{\Y_{q} \rtimes_{\beta^{q}} G}
           }
           {A \rtimes_{\alpha} G}^{\frac{1}{2}}
\end{align*}
and
\begin{align*}
    \Norm{\sum_{i = 1}^{n} \zeta_{i} \eta_{i}}{\Y_{p q} \rtimes_{\beta^{p q}} G}
& = \Norm{\Inner{\sum_{i = 1}^{n} \zeta_{i} \eta_{i}}{\sum_{j = 1}^{n} \zeta_{j} \eta_{j}}{\Y_{p q} \rtimes_{\beta^{p q}} G}}
         {A \rtimes_{\alpha} G}^{\frac{1}{2}} \\
& = \Norm{\sum_{i,j = 1}^{n} \Inner{\zeta_{i} \eta_{i}}{\zeta_{j} \eta_{j}}{\Y_{p q} \rtimes_{\beta^{p q}} G}}
         {A \rtimes_{\alpha} G}^{\frac{1}{2}}.
\end{align*}
To see that
$$
  \Norm{\sum_{i = 1}^{n} \zeta_{i} \otimes \eta_{i}}
       {\Br{\Y_{p} \rtimes_{\beta^{p}} G} \otimes_{A \rtimes_{\alpha} G} \Br{\Y_{q} \rtimes_{\beta^{q}} G}}
= \Norm{\sum_{i = 1}^{n} \zeta_{i} \eta_{i}}{\Y_{p q} \rtimes_{\beta^{p q}} G},
$$
it thus suffices to show that for all $ i,j=1,...,n $ 
$$
\Inner{\eta_{i}}{\Inner{\zeta_{i}}{\zeta_{j}}{\Y_{p} \rtimes_{\beta^{p}} G} \eta_{j}}{\Y_{q} \rtimes_{\beta^{q}} G}
\qquad \text{and} \qquad
\Inner{\zeta_{i} \eta_{i}}{\zeta_{j} \eta_{j}}{\Y_{p q} \rtimes_{\beta^{p q}} G}
$$
are identical elements of $ \Cc{G,A} $. Indeed, for all $ r \in G $,
\begin{align*}
  & ~ \func{\Inner{\eta_{i}}{\Inner{\zeta_{i}}{\zeta_{j}}{\Y_{p} \rtimes_{\beta^{p}} G} \eta_{j}}{\Y_{q} \rtimes_{\beta^{q}} G}}{r} \\
= & ~ \Int{G}
          {
          \Act{\alpha}{u^{- 1}}
              {\Inner{\func{\eta_{i}}{u}}{\Func{\Inner{\zeta_{i}}{\zeta_{j}}{\Y_{p} \rtimes_{\beta^{p}} G} \eta_{j}}{u r}}{\Y_{q}}}
          }
          {u} \\
= & ~ \Int{G}
          {
          \Act{\alpha}{u^{- 1}}
              {
              \Inner{\func{\eta_{i}}{u}}
                    {
                    \Int{G}
                        {\func{\Inner{\zeta_{i}}{\zeta_{j}}{\Y_{p} \rtimes_{\beta^{p}} G}}{t} \Act{\beta^{q}}{t}{\func{\eta_{j}}{t^{- 1} u r}}}
                        {t}
                    }
                    {\Y_{q}}
              }
          }
          {u} \\
= & ~ \Int{G \times G}
          {
          \Act{\alpha}{u^{- 1}}
              {
              \Inner{\func{\eta_{i}}{u}}
                    {\func{\Inner{\zeta_{i}}{\zeta_{j}}{\Y_{p} \rtimes_{\beta^{p}} G}}{t} \Act{\beta^{q}}{t}{\func{\eta_{j}}{t^{- 1} u r}}}
                    {\Y_{q}}
              }
          }
          {\Br{t \times u}} \\
= & ~ \Int{G \times G}
          {
          \Act{\alpha}{u^{- 1}}
              {
              \Inner{\func{\eta_{i}}{u}}
                    {
                    \SqBr{\Int{G}{\Act{\alpha}{s^{- 1}}{\Inner{\func{\zeta_{i}}{s}}{\func{\zeta_{j}}{s t}}{\Y_{p}}}}{s}}
                    \Act{\beta^{q}}{t}{\func{\eta_{j}}{t^{- 1} u r}}
                    }
                    {\Y_{q}}
              }
          }
          {\Br{t \times u}} \\
= & ~ \Int{G \times G \times G}
          {
          \Act{\alpha}{u^{- 1}}
              {
              \Inner{\func{\eta_{i}}{u}}
                    {
                    \Act{\alpha}{s^{- 1}}{\Inner{\func{\zeta_{i}}{s}}{\func{\zeta_{j}}{s t}}{\Y_{p}}}
                    \Act{\beta^{q}}{t}{\func{\eta_{j}}{t^{- 1} u r}}
                    }
                    {\Y_{q}}
              }
          }
          {\Br{s \times t \times u}}
\end{align*}
and
\begin{align*}
  & ~ \func{\Inner{\zeta_{i} \eta_{i}}{\zeta_{j} \eta_{j}}{\Y_{p q} \rtimes_{\beta^{p q}} G}}{r} \\
= & ~ \Int{G}{\Act{\alpha}{u^{- 1}}{\Inner{\Func{\zeta_{i} \eta_{i}}{u}}{\Func{\zeta_{j} \eta_{j}}{u r}}{\Y_{p q}}}}{u} \\
= & ~ \Int{G}
          {
          \Act{\alpha}{u^{- 1}}
              {
              \Inner{\Int{G}{\func{\zeta_{i}}{s} \Act{\beta^{q}}{s}{\func{\eta_{i}}{s^{- 1} u}}}{s}}
                    {\Int{G}{\func{\zeta_{j}}{t} \Act{\beta^{q}}{t}{\func{\eta_{j}}{t^{- 1} u r}}}{t}}
                    {\Y_{p q}}
              }
          }
          {u} \\
= & ~ \Int{G \times G \times G}
          {
          \Act{\alpha}{u^{- 1}}
              {
              \Inner{\func{\zeta_{i}}{s} \Act{\beta^{q}}{s}{\func{\eta_{i}}{s^{- 1} u}}}
                    {\func{\zeta_{j}}{t} \Act{\beta^{q}}{t}{\func{\eta_{j}}{t^{- 1} u r}}}
                    {\Y_{p q}}
              }
          }
          {\Br{s \times t \times u}} \\
= & ~ \Int{G \times G \times G}
          {
          \Act{\alpha}
              {u^{- 1}}
              {
              \Inner{\func{\zeta_{i}}{s} \otimes \Act{\beta^{q}}{s}{\func{\eta_{i}}{s^{- 1} u}}}
                    {\func{\zeta_{j}}{t} \otimes \Act{\beta^{q}}{t}{\func{\eta_{j}}{t^{- 1} u r}}}
                    {\Y_{p} \otimes_{A} \Y_{q}}
              }
          }
          {\Br{s \times t \times u}} \\
= & ~ \Int{G \times G \times G}
          {
          \Act{\alpha}
              {u^{- 1}}
              {\Inner{\Act{\beta^{q}}{s}{\func{\eta_{i}}{s^{- 1} u}}}{\Inner{\func{\zeta_{i}}{s}}{\func{\zeta_{j}}{t}}{\Y_{p}} \Act{\beta^{q}}{t}{\func{\eta_{j}}{t^{- 1} u r}}}{\Y_{q}}}
          }
          {\Br{s \times t \times u}} \\
= & ~ \Int{G \times G \times G}
          {
          \Act{\alpha}
              {u^{- 1} s}
              {\Inner{\func{\eta_{i}}{s^{- 1} u}}{\Act{\alpha}{s^{- 1}}{\Inner{\func{\zeta_{i}}{s}}{\func{\zeta_{j}}{t}}{\Y_{p}}} \Act{\beta^{q}}{s^{- 1} t}{\func{\eta_{j}}{t^{- 1} u r}}}{\Y_{q}}}
          }
          {\Br{s \times t \times u}} \\
  & ~ \Br{\text{By the axioms for a group action on a $ C^{\ast} $-correspondence.}} \\
= & ~ \Int{G \times G \times G}
          {
          \Act{\alpha}
              {u^{- 1} s}
              {
              \Inner{\func{\eta_{i}}{s^{- 1} u}}{\Act{\alpha}{s^{- 1}}{\Inner{\func{\zeta_{i}}{s}}{\func{\zeta_{j}}{s t}}{\Y_{p}}} \Act{\beta^{q}}{t}{\func{\eta_{j}}{t^{- 1} s^{- 1} u r}}}{\Y_{q}}
              }
          }
          {\Br{s \times t \times u}} \\
  & ~ \Br{\text{By the change of variables $ t \mapsto s t $.}} \\
= & ~ \Int{G \times G \times G}
          {
          \Act{\alpha}
              {u^{- 1}}
              {\Inner{\func{\eta_{i}}{u}}{\Act{\alpha}{s^{- 1}}{\Inner{\func{\zeta_{i}}{s}}{\func{\zeta_{j}}{s t}}{\Y_{p}}} \Act{\beta^{q}}{t}{\func{\eta_{j}}{t^{- 1} u r}}}{\Y_{q}}}
          }
          {\Br{s \times t \times u}}. \\
  & ~ \Br{\text{By the change of variables $ u \mapsto s u $.}}
\end{align*}
Hence,
$$
\forall r \in G: \qquad
  \func{\Inner{\eta_{i}}{\Inner{\zeta_{i}}{\zeta_{j}}{\Y_{p} \rtimes_{\beta^{p}} G} \eta_{j}}{\Y_{q} \rtimes_{\beta^{q}} G}}{r}
= \func{\Inner{\zeta_{i} \eta_{i}}{\zeta_{j} \eta_{j}}{\Y_{p q} \rtimes_{\beta^{p q}} G}}{r}
$$
as claimed, so
$$
  \Norm{\sum_{i = 1}^{n} \zeta_{i} \otimes \eta_{i}}{\Br{\Y_{p} \rtimes_{\beta^{p}} G} \otimes_{A \rtimes_{\alpha} G} \Br{\Y_{q} \rtimes_{\beta^{q}} G}}
= \Norm{\func{\Omega_{p,q}}{\sum_{i = 1}^{n} \zeta_{i} \otimes \eta_{i}}}{\Y_{p q} \rtimes_{\beta^{p q}} G}.
$$
As $ \Cc{G,\Y_{p}} \otimes_{\Cc{G,A}} \Cc{G,\Y_{q}} $ is dense in $ \Br{\Y_{p} \rtimes_{\beta^{p}} G} \otimes_{A \rtimes_{\alpha} G} \Br{\Y_{q} \rtimes_{\beta^{q}} G} $, we can conclude that $ \Omega_{p,q} $ extends uniquely to a $ \C $-linear isometry
$$
\overline{\Omega}_{p,q}: \Br{\Y_{p} \rtimes_{\beta^{p}} G} \otimes_{A \rtimes_{\alpha} G} \Br{\Y_{q} \rtimes_{\beta^{q}} G} \to \Y_{p q} \rtimes_{\beta^{p q}} G.
$$

We wish to show that $ \overline{\Omega}_{p,q} $ is $ \Br{A \rtimes_{\alpha} G} $-linear for all $ p,q \in P $, but this will turn out to be a consequence of the following two facts about these maps, which we need to prove in any case:
\begin{itemize}
\item
For $ p \in P $, $ f \in A \rtimes_{\alpha} G $, and $ \zeta \in \Y_{p} \rtimes_{\beta^{p}} G $,
$$
f \zeta = \func{\overline{\Omega}_{e,p}}{f \otimes \zeta}
\qquad \text{and} \qquad
\zeta f = \func{\overline{\Omega}_{p,e}}{\zeta \otimes f}.
$$

\item
For $ p,q,r \in P $, $ \zeta \in \Y_{p} \rtimes_{\beta^{p}} G $, $ \eta \in \Y_{q} \rtimes_{\beta^{q}} G $, and $ \xi \in Y_{r} \rtimes_{\beta^{r}} G $,
$$
  \func{\overline{\Omega}_{p q,r}}{\func{\overline{\Omega}_{p,q}}{\zeta \otimes \eta} \otimes \xi}
= \func{\overline{\Omega}_{p,q r}}{\zeta \otimes \func{\overline{\Omega}_{q,r}}{\eta \otimes \xi}},
$$
which holds because the multiplication law of the product system is associative.
\end{itemize}

To prove the first fact, let $ p \in P $. Then for all $ f \in \Cc{G,A} $, $ \zeta \in \Cc{G,\Y_{p}} $, and $ s \in G $,
\begin{align*}
    \FUNC{\func{\Omega_{e,p}}{f \otimes \zeta}}{s}
& = \Int{G}{\func{\MULT{e}{p}}{\func{f}{t} \otimes \Act{\beta^{p}}{t}{\func{\zeta}{t^{- 1} s}}}}{t} \\
& = \Int{G}{\func{f}{r} \Act{\beta^{p}}{t}{\func{\zeta}{t^{- 1} s}}}{t} \\
& = \Func{f \zeta}{s}.
\end{align*}
By continuity, therefore, $ f \zeta = \func{\overline{\Omega}_{e,p}}{f \otimes \zeta} $ for all $ f \in A \rtimes_{\alpha} G $ and $ \zeta \in \Y_{p} \rtimes_{\beta^{p}} G $, and the same kind of argument establishes that $ \zeta f = \func{\overline{\Omega}_{p,e}}{\zeta \otimes f} $ also.

Now, to see the $ \Br{A \rtimes_{\alpha} G} $-linearity of $ \overline{\Omega}_{p,q} $ for all $ p,q \in P $, simply observe for all $ f \in A \rtimes_{\alpha} G $, $ \zeta \in \Y_{p} \rtimes_{\beta^{p}} G $, and $ \eta \in \Y_{q} \rtimes_{\beta^{q}} G $ that
\begin{align*}
    \func{\overline{\Omega}_{p,q}}{\Br{\zeta \otimes \eta} f}
& = \func{\overline{\Omega}_{p,q}}{\zeta \otimes \eta f} \\
& = \func{\overline{\Omega}_{p,q}}{\zeta \otimes \func{\overline{\Omega}_{q,e}}{\eta \otimes f}} \\
& = \func{\overline{\Omega}_{p q,e}}{\func{\overline{\Omega}_{p,q}}{\zeta \otimes \eta} \otimes f} \\
& = \func{\overline{\Omega}_{p,q}}{\zeta \otimes \eta} f.
\end{align*}
By linearity and continuity, $ \overline{\Omega}_{p,q} $ is therefore $ \Br{A \rtimes_{\alpha} G} $-linear.

Finally, we will prove that $ \overline{\Omega}_{p,q} $ is surjective for all $ p,q \in P $. Firstly, note that for all $ p \in P $ and $ \zeta \in \Cc{G,\Y_{p}} $,
\begin{align*}
       \Norm{\zeta}{\Y_{p} \rtimes_{\beta^{p}} G}
& =    \Norm{\Inner{\zeta}{\zeta}{\Y_{p} \rtimes_{\beta^{p}} G}}{A \rtimes_{\alpha} G}^{\frac{1}{2}} \\
& \leq \Norm{\Inner{\zeta}{\zeta}{\Y_{p} \rtimes_{\beta^{p}} G}}{\Lp{1}{G,A}}^{\frac{1}{2}} \\
& =    \SqBr{\Int{G}{\Norm{\func{\Inner{\zeta}{\zeta}{\Y_{p} \rtimes_{\beta^{p}} G}}{t}}{A}}{t}}^{\frac{1}{2}} \\
& =    \SqBr{\Int{G}{\Norm{\Int{G}{\Act{\alpha}{s^{- 1}}{\Inner{\func{\zeta}{s}}{\func{\zeta}{s t}}{\Y_{p}}}}{s}}{A}}{t}}^{\frac{1}{2}} \\
& \leq \SqBr{\Int{G \times G}{\Norm{\Act{\alpha}{s^{- 1}}{\Inner{\func{\zeta}{s}}{\func{\zeta}{s t}}{\Y_{p}}}}{A}}{\Br{s \times t}}}^{\frac{1}{2}} \\
& =    \SqBr{\Int{G \times G}{\Norm{\Inner{\func{\zeta}{s}}{\func{\zeta}{s t}}{\Y_{p}}}{A}}{\Br{s \times t}}}^{\frac{1}{2}} \\
& \leq \SqBr{\Int{G \times G}{\Norm{\func{\zeta}{s}}{\Y_{p}} \Norm{\func{\zeta}{s t}}{\Y_{p}}}{\Br{s \times t}}}^{\frac{1}{2}} \quad
       \Br{\text{By the Cauchy-Schwarz Inequality.}} \\
& =    \SqBr{\Int{G}{\Br{\Norm{\func{\zeta}{s}}{\Y_{p}} \Int{G}{\Norm{\func{\zeta}{s t}}{\Y_{p}}}{t}}}{s}}^{\frac{1}{2}} \\
& =    \SqBr{\Int{G}{\Norm{\func{\zeta}{s}}{\Y_{p}} \Norm{\zeta}{\Lp{1}{G,\Y_{p}}}}{s}}^{\frac{1}{2}} \\
& =    \SqBr{\Norm{\zeta}{\Lp{1}{G,\Y_{p}}}^{2}}^{\frac{1}{2}} \\
& =    \Norm{\zeta}{\Lp{1}{G,\Y_{p}}}.
\end{align*}
Fix $ p,q \in P $. Clearly, we can approximate a function $ \zeta \in \Cc{G,\Y_{p q}} $ with respect to $ \Norm{\cdot}{\Lp{1}{G,\Y_{p q}}} $ --- and hence with respect to $ \Norm{\cdot}{\Y_{p q} \rtimes_{\beta^{p q}} G} $ --- by a linear combination of functions of the form $ f \odot z $, where $ f \in \Cc{G} $ and $ z \in \Y_{p q} $. As $ \MULT{p}{q}: \Y_{p} \otimes_{A} \Y_{q} \to \Y_{p q} $ is an isomorphism, we can approximate $ z $ itself by a linear combination of elements of $ \Y_{p q} $ of the form $ \func{\MULT{p}{q}}{x \otimes y} $, where $ x \in \Y_{p} $ and $ y \in \Y_{q} $. Now, for any $ \epsilon > 0 $, we can find an open neighborhood $ U $ of $ e_{G} $ and a non-negative function $ g \in \Cc{G,\R} $ with $ \Supp{g} \subseteq U $ and integral $ 1 $ such that
$$
\Norm{f \odot \func{\MULT{p}{q}}{x \otimes y} - \func{\Omega_{p,q}}{\Br{g \odot x} \otimes \Br{f \odot y}}}{\Lp{1}{G,\Y_{p q}}} < \epsilon.
$$
This yields, according to the foregoing discussion,
$$
\Norm{f \odot \func{\MULT{p}{q}}{x \otimes y} - \func{\Omega_{p,q}}{\Br{g \odot x} \otimes \Br{f \odot y}}}{\Y_{p q} \rtimes_{\beta^{p q}} G} < \epsilon.
$$
Therefore, $ \Range{\overline{\Omega}_{p,q}} $ is dense in $ \Y_{p q} \rtimes_{\beta^{p q} G} $, and as $ \overline{\Omega}_{p,q} $ is an isometry between Banach spaces, it follows that $ \overline{\Omega}_{p,q} $ is surjective.

As $ \overline{\Omega}_{p,q} $ is a surjective $ \Br{A \rtimes_{\alpha} G} $-linear isometry for all $ p,q \in P $, we can apply the main result of \cite{L} by Lance to conclude that it is a unitary operator.
\end{proof}



\begin{Thm}
Suppose that a group $ G $ acts on a row-finite and faithful $ P $-indexed product system $ \Y $ over $A$ via automorphisms $ \beta^{p}_{g} $. Then $ G $ acts on $ \O{\Y} $ via automorphisms denoted by $ \gamma_{g} $. Moreover, if $ G $ is amenable, then $ \Y \rtimes_{\beta} G $ is row-finite and faithful, and
$$
\O{\Y} \rtimes_{\gamma} G \cong \O{\Y \rtimes_{\beta} G}.
$$
\end{Thm}

\begin{proof}
Let $ p \in P $. Recall the strongly-continuous action $ \tau^{p} $ of $ G $ on $ \Comp{\Y_{p}} $ given by
$$
\forall x,y \in \Y_{p}: \qquad
\Act{\tau^{p}}{g}{\Theta_{x,y}} = \Theta_{\Act{\beta^{p}}{g}{x},\Act{\beta^{p}}{g}{y}}.
$$
The left-action $ \phi_{p}: A \to \Comp{\Y_{p}} $ is injective by assumption. To see that it is equivariant for $ \alpha $ and $ \tau^{p} $, firstly observe that for all $ g \in G $, $ a \in A $, and $ x \in \Y_{p} $ 
$$
  \Act{\beta^{p}}{g}{\FUNC{\func{\phi_{p}}{a}}{x}}
= \Act{\beta^{p}}{g}{a x}
= \Act{\alpha}{g}{a} \Act{\beta^{p}}{g}{x}
= \FUNC{\func{\phi_{p}}{\Act{\alpha}{g}{a}}}{\Act{\beta^{p}}{g}{x}},
$$
so $ \beta^{p}_{g} \circ \func{\phi_{p}}{a} = \func{\phi_{p}}{\Act{\alpha}{g}{a}} \circ \beta^{p}_{g} $; equivalently, $ \beta^{p}_{g} \circ \func{\phi_{p}}{a} \circ \beta^{p}_{g^{- 1}} = \func{\phi_{p}}{\Act{\alpha}{g}{a}} $. 
Next, observe for all $ g \in G $ and $ x,y,z \in \Y_{p} $ that
\begin{align*}
    \Func{\beta^{p}_{g} \circ \Theta_{x,y} \circ \beta^{p}_{g^{- 1}}}{z}
& = \Act{\beta^{p}}{g}{x \Inner{y}{\Act{\beta^{p}}{g^{- 1}}{z}}{\Y_{p}}} \\
& = \Act{\beta^{p}}{g}{x} \Act{\alpha}{g}{\Inner{y}{\Act{\beta^{p}}{g^{- 1}}{z}}{\Y_{p}}} \\
& = \Act{\beta^{p}}{g}{x} \Inner{\Act{\beta^{p}}{g}{y}}{z}{\Y_{p}} \\
& = \func{\Theta_{\Act{\beta^{p}}{g}{x},\Act{\beta^{p}}{g}{y}}}{z},
\end{align*}
so $ \Act{\tau^{p}}{g}{\Theta_{x,y}} = \beta^{p}_{g} \circ \Theta_{x,y} \circ \beta^{p}_{g^{- 1}} $. In particular, as $ \Range{\phi_{p}} \subseteq \Comp{\Y_{p}} $, we have
$$
\forall a \in A: \qquad
  \Act{\tau^{p}}{g}{\func{\phi_{p}}{a}}
= \beta^{p}_{g} \circ \func{\phi_{p}}{a} \circ \beta^{p}_{g^{- 1}}
= \func{\phi_{p}}{\Act{\alpha}{g}{a}},
$$
which proves the equivariance of $ \phi_{p} $ for $ \alpha $ and $ \tau^{p} $. According to the theory of reduced $ C^{\ast} $-crossed products, $ \phi_{p} $ induces the injective $ \ast $-homomorphism
$$
\overline{\phi_{p}}: A \rtimes_{\alpha,\red} G \to \Comp{\Y_{p}} \rtimes_{\tau^{p},\red} G,
$$
where $ \func{\overline{\phi_{p}}}{f} = \phi_{p} \circ f $ for all $ f \in \Cc{G,A} $. However, $ G $ is amenable, so
$$
\overline{\phi_{p}}: A \rtimes_{\alpha} G \to \Comp{\Y_{p}} \rtimes_{\tau^{p}} G \stackrel{\cong}{\longrightarrow} \Comp{\Y_{p} \rtimes_{\beta^{p}} G},
$$
where the inverse of the $ \ast $-isomorphism on the right is defined in \cite{HN} by
$$
\forall \zeta,\eta \in \Cc{G,\Y_{p}}, ~ \forall s \in G: \qquad
  \FUNC{\func{\Lambda}{\Theta_{\zeta,\eta}}}{s}
= \Int{G}{\func{\Delta}{s^{- 1} r} \Theta_{\func{\zeta}{r},\Act{\beta^{p}}{s}{\func{\eta}{s^{- 1} r}}}}{r},
$$
where $ \Delta $ is the modular function of $ G $. Therefore, $ \Y \rtimes_{\beta} G $ is also a row-finite and faithful product system, as claimed.

Next, we show that there exists a strongly-continuous action $ \gamma $ of $ G $ on $ \O{\Y} $ that satisfies
\begin{equation} \label{The Group Action on the Cuntz-Pimsner Algebra}
\forall g \in G, ~ \forall p \in P, ~ \forall y \in \Y_{p}: \qquad
\Act{\gamma}{g}{\func{j_{\Y}}{y}} = \func{j_{\Y}}{\Act{\beta^{p}}{g}{y}},
\end{equation}
where $ j_{\Y}: \Y \to \O{\Y} $ denotes the universal Cuntz-Pimsner representation. Let $ g \in G $. Then the map $ \Psi_{g}: \Y \to \O{\Y} $ defined by $ \func{\Psi_{g}}{y} \df \func{j_{\Y}}{\Act{\beta^{p}}{g}{y}} $ for all $ p \in P $ and $ y \in \Y_{p} $ is a Cuntz-Pimsner representation of $ \Y $ on $ \O{\Y} $:
\begin{itemize}
\item
For all $ p,q \in P $, $ x \in \Y_{p} $, and $ y \in \Y_{q} $, we have
\begin{align*}
    \func{\Psi_{g}}{x y}
& = \func{j_{\Y}}{\Act{\beta^{p + q}}{g}{x y}} \\
& = \func{j_{\Y}}{\Act{\beta^{p}}{g}{x} \Act{\beta^{q}}{g}{y}} \\
& = \func{j_{\Y}}{\Act{\beta^{p}}{g}{x}} \func{j_{\Y}}{\Act{\beta^{q}}{g}{y}} \\
& = \func{\Psi_{g}}{x} \func{\Psi_{g}}{y}.
\end{align*}

\item
For all $ p \in P $ and $ x,y \in \Y_{p} $, we have
\begin{align*}
    \func{\Psi_{g}}{\Inner{x}{y}{\Y_{p}}}
& = \func{j_{\Y}}{\Act{\alpha}{g}{\Inner{x}{y}{\Y_{p}}}} \\
& = \func{j_{\Y}}{\Inner{\Act{\beta^{p}}{g}{x}}{\Act{\beta^{p}}{g}{y}}{\Y_{p}}} \\
& = \func{j_{\Y}}{\Act{\beta^{p}}{g}{x}}^{\ast} \func{j_{\Y}}{\Act{\beta^{p}}{g}{y}} \\
& = \func{\Psi_{g}}{x}^{\ast} \func{\Psi_{g}}{y}.
\end{align*}

\item
Let $ p \in P $. The foregoing argument tells us that $ \Psi_{g} $ is a Toeplitz representation of $ \Y $ on $ \O{\Y} $, so there exists an extension $ \Psi_{g}^{\Br{p}}: \Comp{\Y_{p}} \to \O{\Y} $ such that
\begin{align*}
\forall x,y \in \Y_{P}: \qquad
    \func{\Psi_{g}^{\Br{p}}}{\Theta_{x,y}}
& = \func{\Psi_{g}}{x} \func{\Psi_{g}}{y}^{\ast} \\
& = \func{j_{\Y}}{\Act{\beta^{p}}{g}{x}} \func{j_{\Y}}{\Act{\beta^{p}}{g}{y}}^{\ast} \\
& = \func{j_{\Y}^{\Br{p}}}{\Theta_{\Act{\beta^{p}}{g}{x},\Act{\beta^{p}}{g}{y}}} \\
& = \func{j_{\Y}^{\Br{p}}}{\Act{\tau^{p}}{g}{\Theta_{x,y}}},
\end{align*}
which implies by continuity that $ \Psi_{g}^{\Br{p}} = j_{\Y}^{\Br{p}} \circ \tau^{p}_{g} $. As we have shown that $ \phi_{p} $ is equivariant for $ \alpha $ and $ \tau^{p} $, we thus have
$$
\forall a \in A: \qquad
  \func{\Psi_{g}^{\Br{p}}}{\func{\phi_{p}}{a}}
= \func{j_{\Y}^{\Br{p}}}{\Act{\tau^{p}}{g}{\func{\phi_{p}}{a}}}
= \func{j_{\Y}^{\Br{p}}}{\func{\phi_{p}}{\Act{\alpha}{g}{a}}}
= \func{j_{\Y}}{\Act{\alpha}{g}{a}}
= \func{\Psi_{g}}{a},
$$
proving that $ \Psi_{g} $ is a Cuntz-Pimsner representation of $ \Y $.
\end{itemize}
By universality, there is thus a $ C^{\ast} $-endomorphism $ S $ on $ \O{\Y} $ such that
$$
\forall p \in P, ~ \forall y \in \Y_{p}: \qquad
\func{S}{\func{j_{\Y}}{y}} = \func{j_{\Y}}{\Act{\beta^{p}}{g}{y}}.
$$
Similarly, there is a $ C^{\ast} $-endomorphism $ T $ on $ \O{\Y} $ such that
$$
\forall p \in P, ~ \forall y \in \Y_{p}: \qquad
\func{T}{\func{j_{\Y}}{y}} = \func{j_{\Y}}{\Act{\beta^{p}}{g^{- 1}}{y}}.
$$
As $ S T = \Id_{\O{\Y}} = T S $, we see that $ S $ is a $ C^{\ast} $-isomorphism, and as $ g $ is arbitrary, there is an action $ \gamma $ of $ G $ on $ \O{\Y} $ that satisfies \eqref{The Group Action on the Cuntz-Pimsner Algebra}. The strong continuity of $ \gamma $ immediately follows from the continuity of $ j_{\Y} $.

We now show that a Cuntz-Pimsner representation $ \psi: \Y \rtimes_{\beta} G \to \O{\Y} \rtimes_{\gamma} G $ exists and that it satisfies
$$
\forall p \in P, ~ \forall \zeta \in \Cc{G,\Y_{p}}: \qquad
\func{\psi_{p}}{\zeta} = j_{\Y} \circ \zeta.
$$
As $ j_{\Y}|_{A}: A \to \O{\Y} $ is a $ \ast $-homomorphism, and as $ \Act{\gamma}{g}{\func{j_{\Y}}{a}} = \func{j_{\Y}}{\Act{\alpha}{g}{a}} $ for all $ a \in A $, we find that $ j_{\Y}|_{A} $ is equivariant for $ \alpha $ and $ \gamma $. Hence, $ j_{\Y}|_{A} $ induces a $ \ast $-homomorphism
$$
\psi_{e}: A \rtimes_{\alpha} G \to \O{\Y} \rtimes_{\gamma} G
$$
such that $ \func{\psi_{e}}{f} = j_{\Y} \circ f $ for all $ f \in \Cc{G,A} $. Let $ p \in P $ and $ \zeta,\eta \in \Cc{G,\Y_{p}} $. Then for all $ s \in G $,
\begin{align*}
    \FUNC{\Br{j_{\Y} \circ \zeta}^{\ast} \Br{j_{\Y} \circ \zeta}}{s}
& = \Int{G}{\func{\Br{j_{\Y} \circ \zeta}^{\ast}}{r} \Act{\gamma}{r}{\Func{j_{\Y} \circ \zeta}{r^{- 1} s}}}{r} \\
& = \Int{G}
        {
        \func{\Delta}{r^{- 1}} \cdot
        \Act{\gamma}{r}{\func{j_{\Y}}{\func{\zeta}{r^{- 1}}}^{\ast}} \Act{\gamma}{r}{\func{j_{\Y}}{\func{\zeta}{r^{- 1} s}}}
        }
        {r} \\
& = \Int{G}{\Act{\gamma}{r^{- 1}}{\func{j_{\Y}}{\func{\zeta}{r}}^{\ast}} \Act{\gamma}{r^{- 1}}{\func{j_{\Y}}{\func{\zeta}{r s}}}}{r} \\
& = \Int{G}{\Act{\gamma}{r^{- 1}}{\func{j_{\Y}}{\func{\zeta}{r}}^{\ast} \func{j_{\Y}}{\func{\zeta}{r s}}}}{r} \\
& = \Int{G}{\Act{\gamma}{r^{- 1}}{\func{j_{\Y}}{\Inner{\func{\zeta}{r}}{\func{\zeta}{r s}}{\Y_{p}}}}}{r} \\
& = \Int{G}{\func{j_{\Y}}{\Act{\alpha}{r^{- 1}}{\Inner{\func{\zeta}{r}}{\func{\zeta}{r s}}{\Y_{p}}}}}{r} \\
& = \func{j_{\Y}}{\Int{G}{\Act{\alpha}{r^{- 1}}{\Inner{\func{\zeta}{r}}{\func{\zeta}{r s}}{\Y_{p}}}}{r}} \\
& = \func{j_{\Y}}{\func{\Inner{\zeta}{\zeta}{\Y_{p} \rtimes_{\beta^{p}} G}}{s}} \\
& = \FUNC{\func{\psi_{0}}{\Inner{\zeta}{\zeta}{\Y_{p} \rtimes_{\beta^{p}} G}}}{s},
\end{align*}
so
\begin{align*}
       \Norm{j_{\Y} \circ \zeta}{\O{\Y} \rtimes_{\gamma} G}
& =    \Norm{\Br{j_{\Y} \circ \zeta}^{\ast} \Br{j_{\Y} \circ \zeta}}{\O{\Y} \rtimes_{\gamma} G}^{\frac{1}{2}} \\
& =    \Norm{\func{\psi_{0}}{\Inner{\zeta}{\zeta}{\Y_{p} \rtimes_{\beta^{p}} G}}}{\O{\Y} \rtimes_{\gamma} G}^{\frac{1}{2}} \\
& \leq \Norm{\Inner{\zeta}{\zeta}{\Y \rtimes_{\beta^{p}} G}}{A \rtimes_{\alpha} G}^{\frac{1}{2}} \\
& =    \Norm{\zeta}{\Y \rtimes_{\beta} G}.
\end{align*}
In light of this norm-inequality, there exists a continuous linear map
$$
\psi_{p}: \Y_{p} \rtimes_{\beta^{p}} G \to \O{\Y} \rtimes_{\gamma} G
$$
such that $ \func{\psi_{p}}{\zeta} = j_{\Y} \circ \zeta $ for all $ \zeta \in \Cc{G,\Y_{p}} $. By combining the various $ \psi_{p} $'s, we get a map $ \psi: \Y \rtimes_{\beta} G \to \O{\Y} \rtimes_{\gamma} G $. The following show that $ \psi $ is a Toeplitz representation:
\begin{itemize}
\item
As seen above, $ \func{\psi_{e}}{\Inner{\zeta}{\zeta}{\Y_{p} \rtimes_{\beta^{p}} G}} = \func{\psi_{p}}{\zeta}^{\ast} \func{\psi_{p}}{\zeta} $ for all $ p \in P $ and $ \zeta \in \Cc{G,\Y_{p}} $.

\item
For all $ p,q \in P $, $ \zeta \in \Y_{p} \rtimes_{\beta^{p}} G $, $ \eta \in \Y_{q} \rtimes_{\beta^{q}} G $, and $ s \in G $,
\begin{align*}
    \FUNC{\func{\psi_{p}}{\zeta} \func{\psi_{q}}{\eta}}{s}
& = \Int{G}{\FUNC{\func{\psi_{p}}{\zeta}}{r} \Act{\gamma}{r}{\FUNC{\func{\psi_{q}}{\eta}}{r^{- 1} s}}}{r} \\
& = \Int{G}{\func{j_{\Y}}{\func{\zeta}{r}} \Act{\gamma}{r}{\func{j_{\Y}}{\func{\eta}{r^{- 1} s}}}}{r} \\
& = \Int{G}{\func{j_{\Y}}{\func{\zeta}{r}} \func{j_{\Y}}{\Act{\beta^{q}}{r}{\func{\eta}{r^{- 1} s}}}}{r} \\
& = \func{j_{\Y}}{\Int{G}{\func{\zeta}{r} \Act{\beta^{q}}{r}{\func{\eta}{r^{- 1} s}}}{r}} \\
& = \func{j_{\Y}}{\Func{\zeta \eta}{s}} \\
& = \FUNC{\func{\psi_{p + q}}{\zeta \eta}}{s},
\end{align*}
so $ \func{\psi_{p}}{\zeta} \func{\psi_{q}}{\eta} = \func{\psi_{p + q}}{\zeta \eta} $.
\end{itemize}
It thus remains to check Cuntz-Pimsner covariance. If
$$
\psi^{\Br{p}}: \Comp{\Y_{p} \rtimes_{\beta^{p}} G} \to \O{\Y} \rtimes_{\gamma} G
$$
denotes the extension of $ \psi_{p} $, then letting $ p \in P $, $ \zeta,\eta \in \Cc{G,\Y_{p}} $, and $ s \in G $,
\begin{align*}
    \FUNC{\func{\psi^{\Br{p}}}{\Theta_{\zeta,\eta}}}{s}
& = \FUNC{\func{\psi_{p}}{\zeta} \func{\psi_{p}}{\eta}^{\ast}}{s} \\
& = \Int{G}{\FUNC{\func{\psi_{p}}{\zeta}}{r} \Act{\gamma}{r}{\FUNC{\func{\psi_p}{\eta}^{\ast}}{r^{- 1} s}}}{r} \\
& = \Int{G}
        {
        \func{j_{\Y}}{\func{\zeta}{r}} \Act{\gamma}{r}{\func{\Delta}{s^{- 1} r} \cdot
        \Act{\gamma}{r^{- 1} s}{\func{j_{\Y}}{\func{\eta}{s^{- 1} r}}^{\ast}}}
        }
        {r} \\
& = \Int{G}{\func{\Delta}{s^{- 1} r} \cdot \func{j_{\Y}}{\func{\zeta}{r}} \Act{\gamma}{s}{\func{j_{\Y}}{\func{\eta}{s^{- 1} r}}^{\ast}}}{r} \\
& = \Int{G}{\func{\Delta}{s^{- 1} r} \cdot \func{j_{\Y}}{\func{\zeta}{r}} \Act{\gamma}{s}{\func{j_{\Y}}{\func{\eta}{s^{- 1} r}}}^{\ast}}{r} \\
& = \Int{G}{\func{\Delta}{s^{- 1} r} \cdot \func{j_{\Y}}{\func{\zeta}{r}} \func{j_{\Y}}{\Act{\beta^{p}}{s}{\func{\eta}{s^{- 1} r}}}^{\ast}}{r}
    \\
& = \Int{G}{\func{\Delta}{s^{- 1} r} \cdot \func{j_{\Y}^{\Br{p}}}{\Theta_{\func{\zeta}{r},\Act{\beta^{p}}{s}{\func{\eta}{s^{- 1} r}}}}}{r} \\
& = \func{j_{\Y}^{\Br{p}}}{\Int{G}{\func{\Delta}{s^{- 1} r} \cdot \Theta_{\func{\zeta}{r},\Act{\beta^{p}}{s}{\func{\eta}{s^{- 1} r}}}}{r}} \\
& = \func{j_{\Y}^{\Br{p}}}{\FUNC{\func{\Lambda}{\Theta_{\zeta,\eta}}}{s}} \\
& = \FUNC{j_{\Y}^{\Br{p}} \circ \func{\Lambda}{\Theta_{\zeta,\eta}}}{s}.
\end{align*}
Hence, $ \func{\psi^{\Br{p}}}{\Theta_{\zeta,\eta}} = j_{\Y}^{\Br{p}} \circ \func{\Lambda}{\Theta_{\zeta,\eta}} $, which means that $ \func{\psi^{\Br{p}}}{T} = j_{\Y}^{\Br{p}} \circ \func{\Lambda}{T} $ for all $ T \in \Comp{\Y \rtimes_{\beta^{p}} G} $. In particular, we have for all $ f \in \Cc{G,A} $ that
\begin{align*}
    \func{\psi^{\Br{p}}}{\func{\overline{\phi_{p}}}{f}}
& = j_{\Y}^{\Br{p}} \circ \func{\Lambda}{\func{\overline{\phi_{p}}}{f}} \\
& = j_{\Y}^{\Br{p}} \circ \func{\Lambda}{\func{\Lambda^{- 1}}{\phi_{p} \circ f}} \\
& = j_{\Y}^{\Br{p}} \circ \phi_{p} \circ f \\
& = j_{\Y} \circ f \\
& = \func{\psi_{e}}{f}.
\end{align*}
Therefore, $ \psi^{\Br{p}} \circ \overline{\phi_{p}} = \psi_{e} $ for all $ p \in P $, which proves that $ \psi $ is Cuntz-Pimsner covariant. By universality, $ \psi $ determines a unique $ \ast $-homomorphism
$$
\psi_{\ast}: \O{\Y \rtimes_{\beta} G} \to \O{\Y} \rtimes_{\gamma} G.
$$

The image of $ \psi_{\ast} $ generates $ \O{\Y} \rtimes_{\gamma} G $, so $ \psi_{\ast} $ is surjective. The injectivity of $ \psi_{\ast} $ follows from the injectivity of $ \psi_{e} $ and the existence of a gauge action of $ \T^{k} $ on $ \O{\Y} \rtimes_{\gamma} G $; see Lemma 3.3.2 in \cite{DKPS} or Corollary 4.14 in \cite{CLSV}.
\end{proof}

\begin{Cor}
Suppose $P=\mathbb N^k$. If $ A $ is AF and each $ \Cstar $-correspondence $ \Y_{n} $ is full and separable, then $ \O{\Y} \rtimes_{\gamma} \T^{k} $ is AF.
\end{Cor}

\begin{proof}
Recall that in this case $ \O{\Y} \rtimes_{\gamma} \T^{k} $ is Morita-Rieffel equivalent with the core $ \ds \O{Y}^{\gamma} \cong \varinjlim_{n \in \N^{k}} \Comp{\Y_{n}} $ and that each $ \Comp{\Y_{n}} $ is Morita-Rieffel equivalent to $ A $.
\end{proof}

\begin{rmk}
Katsoulis obtained similar results in \cite{Ka} for the so-called generalized gauge action on a product system over a semigroup $P$ which is the positive cone of an abelian group. Moreover, using a Fourier transform, he proves a Takai duality result and generalizes some results of Schafhauser from \cite{Sch}.
\end{rmk}

\begin{Eg}
Let $ G $ be a compact  group and let $ \rho_{1},\ldots,\rho_{k} $ be finite-dimensional representations of $ G $ on Hilbert spaces $\mathcal H_1,...,\mathcal H_k$. Let $\Y$ be the product system with fibers $\Y_n=\mathcal H^n=\mathcal H_1^{\otimes n_1}\otimes\cdots\otimes  \mathcal H_k^{\otimes n_k}$ for $n=(n_1,...,n_k)\in \N^k$, in particular, $A=\Y_0=\mathbb C$. 

The compact group $G$ acts on each fiber $\Y_n$ of the  product system $\Y$ via the representation $\rho^n=\rho_1^{\otimes n_1}\otimes\cdots\otimes \rho_k^{\otimes n_k}$. This action is compatible with the multiplication maps and commutes with the gauge action of $\T^k$.  The crossed product $\Y \rtimes G$ becomes a row-finite and faithful product system indexed by $\mathbb N^k$ over the group $C^*$-algebra $C^*(G)$. Moreover, 
\[\O{\Y} \rtimes G \cong \O{\Y \rtimes G}.\]

The Doplicher-Roberts algebra $\mathcal O_{\rho_1,...,\rho_k}$ constructed in \cite{D} from intertwiners Hom$(\rho^n,\rho^m)$ is isomorphic to the fixed point algebra $\O{\Y}^G$ and is  Morita-Rieffel equivalent to $\O{\Y}\rtimes G$.

\end{Eg}

\begin{Eg}
If a locally compact group $G$ acts on a $k$-graph $\Lambda$ by automorphisms, then $G$ acts on the product system $\Y$ constructed from $\Lambda$ as in Example \ref{k-graph} and the $C^*$-algebra of the product system $\Y\rtimes G$ is isomorphic to $C^*(\Lambda)\rtimes G$. In \cite{FPS} the authors consider the particular case when $G=\mathbb Z^\ell$ and they construct a $(k+\ell)$-graph $\Lambda\times \mathbb Z^\ell$ such that $C^*(\Lambda\times \mathbb Z^\ell)\cong C^*(\Lambda)\rtimes\mathbb Z^\ell$. Our result gives a new perspective of this situation.
\end{Eg}




\end{document}